\documentclass{article}
\usepackage{amsmath, amscd, amssymb, latexsym, epsfig, color, amsthm, mathtools, tikz-cd, array, adjustbox, bbm, tikz, enumitem, relsize}

\makeatletter
\newtheorem*{rep@theorem}{\rep@title}
\newcommand{\newreptheorem}[2]{%
	\newenvironment{rep#1}[1]{%
		\def\rep@title{#2 \ref{##1}}%
		\begin{rep@theorem}}%
		{\end{rep@theorem}}}
\makeatother

\numberwithin{equation}{section}
\newreptheorem{theorem}{Theorem}
\newtheorem{theorem}{Theorem}[section]
\newtheorem{proposition}[theorem]{Proposition}

\newtheorem{corollary}[theorem]{Corollary}
\newtheorem{conjecture}[theorem]{Conjecture}
\newtheorem{lemma}[theorem]{Lemma}

\theoremstyle{definition}

\DeclareMathOperator\lk{\mathrm{lk}}

\DeclareMathOperator\cost{\mathrm{cost}}
\DeclareMathOperator{\Hilb}{\mathrm{Hilb}}
\DeclareMathOperator{\Coker}{\mathrm{Coker}}
\DeclareMathOperator{\supp}{\mathrm{supp}}

\newcommand{\field}{\mathbbm{k}}

\newcommand{\Z}{{\mathcal Z}}
\newcommand{\ZZ}{{\mathbb Z}}

\newcommand{\C}{{\mathcal C}}
\newcommand{\K}{{\mathcal K}}
\newcommand{\M}{{\mathcal M}}
\newcommand{\N}{{\mathcal N}}

\newcommand{\Bb}{{\mathcal B}}
\newcommand{\LL}{{\mathcal L}}

\newcommand{\D}{{\mathcal D}}
\newcommand{\alg}{{\mathfrak{a}}}
\newcommand{\sig}{{\mathfrak{s}}}

\newcommand{\im}{\operatorname{Im}}

\newcommand{\mideal}{\ensuremath{\mathfrak{m}}}

\newcommand{\Ker}{\ensuremath{\mathrm{Ker}}\hspace{1pt}}

\tikzset{
	labl/.style={anchor=south, rotate=90, inner sep=.5mm}
}

\title{Algebraic $h$-vectors of simplicial complexes through local cohomology, part 1}
\author{Connor Sawaske\\
	\small \texttt{connor@sawaske.com}
}

\begin{document}
	%%%%%%%%%%%%%%%%%%%%%%%%%%%%%%%%%%%%%%%%%%%%
	%%%%%%%%%%%%%%%%%%%%%%%%%%%%%%%%%%%%%%%%%%%%
	\maketitle
	
	\begin{abstract}
	Given an infinite field $\field$ and a simplicial complex $\Delta$, a common theme in studying the $f$- and $h$-vectors of $\Delta$ has been the consideration of the Hilbert series of the Stanley--Reisner ring $\field[\Delta]$ modulo a generic linear system of parameters $\Theta$. Historically, these computations have been restricted to special classes of complexes (most typically triangulations of spheres or manifolds). We provide a compact topological expression of $h_{d-1}^\alg(\Delta)$, the dimension over $\field$ in degree $d-1$ of $\field[\Delta]/(\Theta)$, for any complex $\Delta$ of dimension $d-1$. In the process, we provide tools and techniques for the possible extension to other coefficients in the Hilbert series.
	\end{abstract}

\section{Introduction}\label{sect:intro}
Associated to every finite simplicial complex $\Delta$ is the notion of its $h$-vector, which is one way of encoding the number of faces that $\Delta$ has in each dimension. Perhaps the most widely-studied combinatorial invariant of a simplicial complex since its inception, properties of this vector continue to be a motivating force in research to this day (see, for example, the continued work toward proving McMullen's long-standing $g$-conjecture in \cite{Adip}).

Given a fixed infinite field $\field$, one of the most powerful tools available for the study of the $h$-vector of $\Delta$ is the Stanley--Reisner ring $\field[\Delta]$. When considered as a vector space over $\field$, this quotient of a polynomial ring has a basis given by monomials whose supports correspond to the faces of $\Delta$. Hence, computations of the $h$-vector can often be reduced to counting dimensions of graded pieces of $\field[\Delta]$ over $\field$.

We introduce two brief non-standard pieces of notation: if $\Theta$ is a linear system of parameters (or l.s.o.p.) for $\field[\Delta]$ and $\Sigma(\Theta; \field[\Delta])$ is the sigma submodule which it generates (definitions to be given later), then let
\[
h_i^\alg(\Delta):=\dim_\field\left(\frac{\field[\Delta]}{(\Theta)}\right)_i\,\,\text{ and }\,\,\,
h_i^\sig(\Delta):=\dim_\field\left(\frac{\field[\Delta]}{\Sigma(\Theta; \field[\Delta])}\right)_i.
\]
The $h^\alg$- and $h^\sig$-vectors are invariant under a generic choice of parameters when $\field$ is infinite, and hence are defined without respect to any particular system $\Theta$. In the literature, these dimensions have usually been written as (or shown to be equivalent to) $h_i'(\Delta)$ and $h_i''(\Delta)$, respectively, when considering triangulated spheres or manifolds.

When $\Delta$ is Cohen--Macaulay (for example, a triangulation of sphere), the following theorem due to Stanley demonstrates the powerful use of techniques from commutative algebra to produce a beautiful correspondence between the $h$-vector of $\Delta$ and the Hilbert series of $\field[\Delta]$. Note that the Cohen--Macaulay property ensures that the submodules $\Sigma(\Theta; \field[\Delta])$ and $(\Theta)$ are equal, and hence $h_i^\alg(\Delta) = h_i^\sig(\Delta)$ for all $i$.

\begin{theorem}{\rm \cite[Section II]{St-96}}\label{Stanley}
	Let $\Delta$ be a $(d-1)$-dimensional Cohen-Macaulay simplicial complex and let $\Theta$ be an l.s.o.p.~for $\field[\Delta]$. Then
	\[
		h_i^\alg(\Delta) = h_i(\Delta)
	\]
	for $i=0, \ldots, d$.
\end{theorem}

When moving from spheres to manifolds, it turns out that there is some discrepancy between $h_i^\alg(\Delta)$ and $h_i(\Delta)$ which can be measured by topological invariants of $\Delta$. For the next results, let $\widetilde{\beta}_i(\Delta)$ denote the $i$-th reduced Betti number of $\Delta$ computed over $\field$. In order to more succinctly express some pre-existing results as well as our own contributions, we now introduce the notion of the truncated reduced Euler characteristic of $\Delta$, denoted by
\begin{equation}\label{truncdef}
\widetilde{\chi}_i(\Delta):=\sum_{j=-1}^i(-1)^j\widetilde{\beta}_j(\Delta).
\end{equation}

Schenzel generalized Stanley's equality in Theorem \ref{Stanley} to the case of triangulations of closed manifolds (or more generally, Buchsbaum complexes) by applying Hochster's results connecting local cohomology modules of Stanley--Reisner rings to the topology of $\Delta$, producing the following theorem.

\begin{theorem}\label{Schenzel}{\rm\cite[p.~137]{Schenzel}}
	If $\Delta$ is a $(d-1)$-dimensional Buchsbaum complex and $\Theta$ is an l.s.o.p.~for $\field[\Delta]$, then
	\[
	h_i^\alg(\Delta) = h_i(\Delta)+(-1)^i\binom{d}{i}\widetilde{\chi}_{i-2}(\Delta)
	\]
	for $i=0, \ldots, d$.
\end{theorem}

Much more recently, Murai, Novik, and Yoshida were able to build upon the equality in the above theorem through the use of the sigma submodule $\Sigma(\Theta; \field[\Delta])$, providing a calculation of the reduced algebraic $h$-vector in the next theorem.

\begin{theorem}\label{BBMD}{\rm\cite[Theorem 1.2]{MNY}}
	If $\Delta$ is a $(d-1)$-dimensional Buchsbaum complex and $\Theta$ is an l.s.o.p.~for $\field[\Delta]$, then
	\[
	h_i^\sig(\Delta) = h_i(\Delta)+(-1)^i\binom{d}{i}\widetilde{\chi}_{i-1}(\Delta).
	\]
\end{theorem}

The Cohen--Macaulay and Buchsbaum hypotheses of these theorems can be very restrictive, and the tools involved in their proofs rely crucially upon powerful algebraic implications of these properties. As a result, few extensions to more general cases have been provided and those that do exist are still fairly specialized (for example, see \cite{NS-sing} for a treatment of complexes with isolated singularities). Our present overarching goal is to provide full computations of the entries in the algebraic $h$-vectors for \textit{any} complex $\Delta$. Aside from low-dimensional special cases, at this time only the top entry, $h_{d}^\alg(\Delta)$, for $\Delta$ of arbitrary dimension $d-1$ has been computed in the following result due to Tay, White, and Whitely in \cite[Theorem 4.1]{TWW} as well as Babson and Novik in \cite[Lemma 2.2(3)]{Nongeneric}.

\begin{theorem}
	If $\Delta$ is a $(d-1)$-dimensional complex, then $h_d^\alg(\Delta)=\widetilde{\beta}_{d-1}(\Delta)$.
\end{theorem}

In this paper, we focus on a computation of $h_{d-1}^\alg(\Delta)$ and $h_{d-1}^\sig(\Delta)$ for an arbitrary $(d-1)$-dimensional complex $\Delta$. However, perhaps the more valuable contribution of our results is the method of their proof, which appears to provide an avenue for the full computation of algebraic $h$-vectors of arbitrary complexes. These results are optimistically planned for a follow-up paper. For the sake of brevity, we will state here only the equation concerning the reduced algebraic $h$-vector, $h^\sig(\Delta)$, and refer the reader to Theorem \ref{hd1Thm} for the non-reduced version. In this statement, $\lk_\Delta F$ denotes the link of a face $F$ in $\Delta$.

\begin{reptheorem}{sigmamodThm}
	Let $\Delta$ be a $(d-1)$-dimensional simplicial complex. Then
	\[
	h_{d-1}^\sig(\Delta) = h_{d-1}(\Delta) + (-1)^{d-1}\sum_{F\in \Delta}\binom{d-|F|}{d-1}\widetilde{\chi}_{d-2-|F|}(\lk_\Delta F).
	\]
\end{reptheorem}

Note that since $\widetilde{\beta}_i(\lk_\Delta F) = 0$ for all faces $F\not=\emptyset$ and all $i<d-1-|F|$ when $\Delta$ is Buchsbaum and that $\lk_\Delta\emptyset = \Delta$, this expression is consistent with Theorem \ref{BBMD}. Furthermore, due to the $d-1$ term in the binomial coefficient, the only faces that may potentially contribute a non-zero term in the sum are those of size $0$ or $1$. Despite this, the theorem has been written in the more general form summing over all $F\in\Delta$ to provide both a compact expression as well as a goalpost for further equalities proposed in Conjecture \ref{mainconjecture}. 

 As a final note, Theorem \ref{sigmamodThm} should not be too surprising when compared to Theorems \ref{Stanley} and \ref{BBMD}. Indeed, a $(d-1)$-dimensional complex $\Delta$ is Buchsbaum but not Cohen--Macaulay precisely when $\widetilde{\beta}_i(\Delta)\not=0$ for some $i<d-1$ while $\widetilde{\beta}_i(\lk_\Delta F )=0$ for all faces $F\not=\emptyset$ and all $i<d-1-|F|$. In such a case, the difference between $h_i^\sig(\Delta)$ and $h_i(\Delta)$ is written purely in terms of the values of $\widetilde{\beta}_j(\Delta)$ for $j<i$. Hence, in the case when $\Delta$ fails to be Buchsbaum by having $\widetilde{\beta}_i(\lk_\Delta F)\not=0$ for some $F\not=\emptyset$ and $i<d-1-|F|$, it stands to reason that $h_i(\Delta) - h_i^\alg(\Delta)$ may be written purely in terms of the Betti numbers of the links of those faces which contain non-trivial homology in the appropriate dimensions.

The structure of this paper is as follows. In Section \ref{sect:prelim}, we review definitions and the vital connections between combinatorics, topology, and algebra that will provide the tools for our computations. In Section \ref{sect:results}, we prove our main results after introducing some new lemmas rooted in commutative and homological algebra. Section \ref{sect:apps} is devoted to a brief application of our results to suspensions of Buchsbaum complexes. Finally, in Section \ref{sect:conclusion} we will discuss the implications of our results by presenting possible generalizations and alternate viewpoints.

\section{Preliminaries}\label{sect:prelim}

For an excellent overview of many of the definitions and results in this section, we refer the reader to \cite{St-96}.

\subsection{Combinatorics}
Let $V$ be a finite set and let $\field$ be a fixed infinite field. A \textbf{simplicial complex} $\Delta$ with vertex set $V$ is a collection of subsets of $V$ that is closed under inclusion. We call each element $F\in \Delta$ a \textbf{face} of $\Delta$, and in the case that $F$ consists of a single vertex $v\in V$, we often abbreviate $\{v\}$ to $v$. Each face $F\in\Delta$ has a \textbf{dimension} defined by $\dim (F): =|F|-1$. Similarly, the dimension of $\Delta$ is defined by $\dim(\Delta):=\max\{\dim (F): F\in\Delta\}$. If all maximal faces of $\Delta$ under inclusion have the same dimension, then we say that $\Delta$ is \textbf{pure}.

For the remainder of this section, let $\Delta$ be a simplicial complex of dimension $d-1$. If $F$ is a face of $\Delta$, then the \textbf{link} and \textbf{contrastar} of $F$ in $\Delta$ are the two induced simplicial complexes defined by
\[
\lk_\Delta F :=\{G\in \Delta:F\cup G\in\Delta\text{ and }F\cap G = \emptyset\}
\]
and
\[
\cost_\Delta F :=\{G\in\Delta: F\not\subset G\},
\]
respectively.

An important combinatorial invariant associated to $\Delta$ is its $\mathbf{f}$\textbf{-vector}, written in the form $f(\Delta) = (f_{-1}(\Delta), f_0(\Delta), \ldots, f_{d-1}(\Delta))$ with
\[
f_i(\Delta):=|\{F\in\Delta: \dim(F) = i\}|.
\]
Instead of studying the $f$-vector directly, we study what we will refer to as the (combinatorial) \textbf{$h$-vector} of $\Delta$, written as $h(\Delta) = (h_0(\Delta), h_1(\Delta), \ldots, h_d(\Delta))$ with
\[
h_i(\Delta):=\sum_{j=0}^i(-1)^{i-j}\binom{d-j}{i-j}f_{j-1}(\Delta).
\]

\subsection{Topology}

If $\Gamma$ is another simplicial complex with $\Gamma\subset\Delta$, we denote by $H^i(\Delta, \Gamma)$ the $\mathbf{i}$\textbf{-th relative cohomology group} of the pair $(\Delta, \Gamma)$ (over $\field$). In the case that $\Gamma = \emptyset$, we abbreviate $H^i(\Delta, \emptyset)$ to $\widetilde{H}^i(\Delta)$ and refer to this as the $\mathbf{i}$\textbf{-th reduced cohomology group} of $\Delta$ (over $\field$). These cohomology groups are also vector spaces over $\field$, and one of the most important invariants that we focus on is their dimension. The $\mathbf{i}$\textbf{-th reduced Betti number} of $\Delta$ (over $\field$) is defined by
\[
\widetilde{\beta}_i(\Delta) := \dim_\field \widetilde{H}^i(\Delta)
\]
and the $\mathbf{i}$\textbf{-th relative Betti number} of the pair $(\Delta, \Gamma)$ (over $\field$) is defined by
\[
\beta_i(\Delta, \Gamma) := \dim_\field H^i(\Delta, \Gamma).
\]
When $\Gamma=\cost_\Delta F$ for some face $F\in\Delta$, we have natural isomorphisms
\begin{equation}\label{costtolink}
\widetilde{H}^i(\lk_\Delta F)\cong H^{i-|F|}(\Delta, \cost_\Delta F)
\end{equation}
for all $i$ provided by \cite[Lemma]{Grabe}.

We say that $\Delta$ is \textbf{Buchsbaum} if $\Delta$ is pure and $\beta_i(\Delta, \cost_\Delta F)=0$
for all $i<d-1$ and for every face $F\in \Delta$ with $|F|>0$. Similarly, we say that $\Delta$ is \textbf{Cohen-Macaulay} if $\beta_i(\Delta, \cost_\Delta F)=0$ for all $i<d-1$ and for all faces $F\in\Delta$ (including the empty face). Though the primary results in this paper will not specifically involve Buchsbaum or Cohen-Macaulay complexes, many of the computations that we extend and were mentioned in Section \ref{sect:intro} have historically been restricted to these classes of complexes.

In further uses of Betti numbers, we define the (reduced) \textbf{Euler characteristic} of $\Delta$ and the $\mathbf{i}$\textbf{-th truncated Euler characteristic} of $\Delta$ by
\[
\widetilde{\chi}(\Delta) = \sum_{j=-1}^{d-1}(-1)^j\widetilde{\beta}_j(\Delta)\hspace{3pt}\text{ and }\hspace{3pt}
\widetilde{\chi}_i(\Delta) := \sum_{j=-1}^i(-1)^j\widetilde{\beta}_j(\Delta),
\]
respectively, for $i=-1, \ldots, d-1$.

\subsection{Algebra}

Let $A$ be the polynomial ring $\field[x_v: v\in V]$, and let $\mideal$ be the ideal $(x_v: v\in V)$. If $F\subset V$, then denote
\[
x_F:=\prod_{v\in F}x_v.
\]
We define the \textbf{Stanley--Reisner ideal} of $\Delta$ by
\[
I_\Delta:=(x_F: F\not\in \Delta)
\]
and the \textbf{Stanley--Reisner ring} of $\Delta$ by the quotient
\[
\field[\Delta]:=A/I_\Delta.
\]
We can view $A$ both as a $\ZZ$-graded ring by setting $\deg(x_v)=1$ for all $v\in V$ and as a $\ZZ^V$-graded ring by setting $\deg(x_v)=\mathbf{e}_v$, where $\mathbf{e}_v$ is the standard basis element of $\ZZ^V$ corresponding to $v\in V$. When considering some degree $\alpha=(\alpha_v: v\in V)\in\ZZ^V$, we define the \textbf{support} of $\alpha$ by $\supp(\alpha) = \{v:\alpha_v\not=0\}$. Since $I_\Delta$ is homogeneous with respect to either of these gradings, we will consider $\field[\Delta]$ at times as either a $\ZZ$-graded or $\ZZ^V$-graded $A$-module or vector space over $\field$. In general, for a $\ZZ$-graded $A$-module $M$ and $j\in\ZZ$, we denote by $M[j]$ the module obtained from $M$ by shifting degrees by $j$, defined such that $M[j]_i = M_{i+j}$.

Let $V$ be a $\ZZ$-graded vector space over $\field$ such that $V_i$ is finite-dimensional for all $i$. We abbreviate
\[
\D(V_i):=\dim_\field V_i,
\]
and in the case that $V_i = 0$ for $i<0$, we define the \textbf{Hilbert series} of $V$ in the indeterminate $t$ by
\[
\Hilb(V, t):=\sum_{i\ge 0}\D(V_i)t^i.
\]
One of the primary connections between the combinatorics of $\Delta$ and the algebraic properties of $\field[\Delta]$ is the following theorem due to Stanley (see \cite[Section II.2]{St-96}).

\begin{theorem}\label{HilbSeries}
	Let $\Delta$ be a $(d-1)$-dimensional simplicial complex. Then
	\[
	\Hilb(\field[\Delta], t)= \frac{\sum_{i=0}^dh_i(\Delta)t^i}{(1-t)^d}.
	\]
\end{theorem}

One consequence of the above theorem is that $\field[\Delta]$ has Krull dimension $d$. Given any $A$-module $M$ of Krull dimension $d$, we call a sequence $\Theta = (\theta_1, \ldots, \theta_d)$ of linear forms in $A$ a \textbf{linear system of parameters} (or l.s.o.p.) for $M$ if $M/(\Theta) M$ is a finite-dimensional vector space over $\field$. In the case that $\field$ is infinite, any generic choice of $d$ linear forms will satisfy this condition. Given an l.s.o.p.~for $M$, further define the \textbf{Sigma submodule} of $M$ with respect to $\Theta$ by
\[
\Sigma(\Theta; M) := \Theta M + \sum_{i=0}^d (\theta_1, \ldots, \theta_{i-1}, \theta_{i+1}, \ldots, \theta_d)M:_M \theta_i,
\]
where
\[
(\theta_1, \ldots, \theta_{i-1}, \theta_{i+1}, \ldots, \theta_d)M:_M \theta_i = \{m\in M: \theta_i\cdot m \in (\theta_1, \ldots, \theta_{i-1}, \theta_{i+1}, \ldots, \theta_d)M \}.
\]

We are now ready to provide full definitions for the two main invariants of study in this paper. One of them is what we will call the \textbf{algebraic $h$-vector} of $\Delta$, written as $h^\alg(\Delta)= (h_0^\alg(\Delta), \ldots, h_d^\alg(\Delta))$, where
\[
h_i^\alg(\Delta) = \dim_\field (\field[\Delta]/(\Theta))_i=\D\left(\left(\field[\Delta]/(\Theta)\right)_i\right)
\]
for $i=0, 1, \ldots, d$, and the other is the \textbf{reduced algebraic $h$-vector} of $\Delta$, written as $h^\sig(\Delta)= (h_0^\sig(\Delta), \ldots, h_d^\sig(\Delta))$, where
\[
h_i^\sig(\Delta) = \dim_\field (\field[\Delta]/\Sigma(\Theta; M))_i=\D\left(\left(\field[\Delta]/\Sigma(\Theta; M)\right)_i\right).
\]

As usual, our study of algebraic $h$-vectors will involve much consideration of local cohomology modules of Stanley--Reisner rings. For a general introduction to these modules, the reader is referred to \cite{24hours}. In order to clarify and make use of the structure of the local cohomology modules $H_\mideal^i(\field[\Delta])$ from a topological perspective, it will be important to consider them in the $\ZZ^V$-graded setting. Note that regardless of the choice of grading, the overall structures of these modules (and more importantly, the dimensions of their graded pieces) remain unchanged when considering either grading. The following stunning theorem due to Gr\"{a}be in \cite[Theorem 2]{Grabe} will provide the necessary connections between homological properties of $\field[\Delta]$ and the topology of $\Delta$.

\begin{theorem}\label{Grabe}
	Let $\Delta$ be a simplicial complex, and let $\alpha\in\ZZ^V$. If $\alpha\not\in\ZZ_{\le 0}^V$ or $\supp(\alpha)\not\in\Delta$, then $H_\mideal^i(\field[\Delta])_\alpha=0$. Otherwise, writing $\supp(\alpha)=F\in\Delta$,
	\[
	H_\mideal^i(\field[\Delta])_\alpha\cong H^{i-1}(\Delta, \cost_\Delta F)
	\]
	as vector spaces over $\field$. Furthermore, the $A$-module structure of $H_\mideal^i(\field[\Delta])$ can be defined as follows. If $v\not\in F$, then the multiplication map $\cdot x_v$ is the zero map on $H_\mideal^i(\field[\Delta])_\alpha$. If $v\in F$, then
	\[
	\cdot x_v: H_\mideal^i(\field[\Delta])_\alpha\to H_\mideal^i(\field[\Delta])_{\alpha+\mathbf{e}_v}
	\]
	corresponds to the map
	\[
	H^{i-1}(\Delta, \cost_\Delta F))\to H^{i-1}(\Delta, \cost_\Delta (\supp(\alpha+\mathbf{e}_v)))
	\]
	induced by the inclusion of pairs $(\Delta, \cost_\Delta (\supp(\alpha+\mathbf{e}_v)))\to (\Delta, \cost_\Delta F))$.
\end{theorem}

\section{Algebraic $h$-vectors}\label{sect:results}
For the entirety of this section, we will assume that $\Delta$ is an arbitrary $(d-1)$-dimensional complex and that $\Theta = (\theta_1, \ldots, \theta_d)$ is a generic l.s.o.p.~for $\field[\Delta]$. Given a subset $S\subset\Theta$, for the ease of notation we will also use $S$ to denote the submodule $(S)\field[\Delta]$ generated by the ideal $(S)$. In addition, denote
	\[
	\M^i(S) := H_\mideal^i(\field[\Delta]/S).
	\]
If $\theta_j\not\in S$, then also define
	\[
	\K^i(S, \theta_j):=\Ker \left(\cdot \theta_j:\M^i(S)\to \M^i(S)\right)
	\]
and
	\[
	\C\K^i(S, \theta_j):=\Coker \left(\cdot \theta_j:\M^i(S)\to \M^i(S)\right).
	\]
	In the frequent case that $S=(\theta_1, \ldots, \theta_{j-1})$ for some $j$, we make the further abbreviations
	\[
	\M^i(j-1) := \M^i(S) = H_\mideal^i(\field[\Delta]/(\theta_1, \ldots, \theta_{j-1})),
	\]
	\[
	\K^i(j) := \K^i(S, \theta_j) = \Ker \left(\cdot \theta_j: \M^i(j-1)\to \M^i(j-1)\right),
	\]
and
	\[
	\C\K^i(j) := \C\K^i(S, \theta_j) = \Coker \left(\cdot \theta_j:\M^i(j-1)\to \M^i(j-1)\right).
	\]
One particular submodule of $H_\mideal^i(\field[\Delta])$ will play a very important role, which we define by
	\[
	\LL^i_{j}:=\left[\M^i(\emptyset)_{\ge -1}\right]\bigcap_{p=1}^j \left[\Ker \left(\cdot \theta_p:\M^i(\emptyset)_{\ge-1}\to \M^i(\emptyset)_{\ge 0}\right)\right].
	\]
This submodule has a topological interpretation using Gr\"{a}be's theorem. As a $\ZZ$-graded vector space, $\LL^i_j$ is concentrated in degrees $-1$ and $0$. In the top degree, $\LL_j^i$ is isomorphic to $\widetilde{H}^{i-1}(\Delta)$. In degree $-1$, $\LL^i_j$ is isomorphic to the intersection of the kernels of generic linear combinations of maps of the form
\[
H^{i-1}(\Delta, \cost_\Delta v) \to H^{i-1}(\Delta, \emptyset)
\]
induced by inclusions.

\subsection{Calculating $h_{d-1}^\alg(\Delta)-h_{d-1}(\Delta)$}

Before beginning to study $h_{d-1}^\alg(\Delta)$ directly, we introduce four lemmas that will allow for some generalizations of the standard techniques for studying $h^\alg$-vectors of Cohen--Macaulay and Buchsbaum complexes.

\begin{lemma}\label{primeavoidance}
	If $\Delta$ is a $(d-1)$-dimensional simplicial complex and $\Theta = \theta_1, \ldots, \theta_d$ is a generic l.s.o.p.~for $\field[\Delta]$, then
		\[
			\Ker \left(\cdot \theta_j: \field[\Delta]/S\to \field[\Delta]/S\right) = \K^0(S, \theta_j)
		\]
	for any $S\subsetneq\Theta$ and $\theta_j\not\in S$.
\end{lemma}

\begin{proof}
	Define
		\[
			\N(S):= \frac{\field[\Delta]/S}{\M^0(S)}.
		\]
	Since $\field[\Delta]/S$ has Krull dimension $d-|S|$ while $\M^0(S)$ is an Artinian submodule, $\N(S)$ has positive Krull dimension for $0\le |S| \le d-1$. Furthermore, since $\mideal\cdot \N(S)\not=\N(S)$ and $H_\mideal^0(\N(S))=0$, the depth of $\N(S)$ is at least $1$, and so we may assume (under the genericness of $\Theta$) that $\theta_j$ is a non-zero-divisor on $\N(S)$. Hence,
		\[
			\Ker \left(\cdot \theta_j: \field[\Delta]/S\to \field[\Delta]/S\right) = \Ker \left(\cdot \theta_j: \M^0(S)\to \M^0(S)\right).
		\]
\end{proof}

\begin{lemma}\label{HilbertSeries}
	If $\Delta$ is a $(d-1)$-dimensional simplicial complex and $\Theta = \theta_1, \ldots, \theta_d$ is a generic l.s.o.p. for $\field[\Delta]$, then
		\[
			\Hilb(\field[\Delta]/\Theta, t) =\sum_{i=0}^d h_i^\alg(\Delta)t^i = \sum_{i=0}^d h_i(\Delta)t^i + \sum_{j=1}^d(1-t)^{d-j}t\Hilb(\K^0(j), t).
		\]
\end{lemma}

\begin{proof}
	By the standard Hilbert series computations, if
		\[
			\Z(j) = \Ker(\cdot\theta_j: \field[\Delta]/	(\theta_1, \ldots, \theta_{j-1})\to \field[\Delta](\theta_1, \ldots, \theta_{j-1})),
		\]
	then
		\[
			\Hilb(\field[\Delta]/\Theta, t) = \sum_{i=0}^d h_i(\Delta)t^i + \sum_{j=1}^d(1-t)^{d-j}t\Hilb(\Z(j), t).
		\]
	The result now follows by Lemma \ref{primeavoidance}.
\end{proof}

In the case of a Buchsbaum complex, the trivial structure of the local cohomology modules of $\field[\Delta]$ and special properties of $\Theta$ allow for very straightforward computations of the $\K^0(j)$ submodules above using short exact sequences. In the next two lemmas, we use the same prime avoidance technique from Lemma \ref{primeavoidance} to show that a similar approach is still viable in a more general setting.

\begin{lemma}\label{LESLemma}
	If $\Delta$ is a $(d-1)$-dimensional simplicial complex and $\theta_1, \ldots, \theta_d$ is a generic l.s.o.p. for $\field[\Delta]$, then for any $S\subsetneq\Theta$ and $\theta_j\not\in S$ there exists a long exact sequence of graded $A$-modules of the form
		\[
		\begin{tikzpicture}[scale = 1, descr/.style={fill=white,inner sep=1.5pt}]
		\matrix (m) [
		matrix of math nodes,
		row sep=1.5em,
		column sep=1em,
		text height=1ex, text depth=0.25ex
		]
		{ 0 & \frac{\M^0(S)}{\theta_j\cdot (\M^0(S)[-1])} & \M^0(S\cup\{\theta_j\}) \\
			\M^1(S)[-1] & \M^1(S) & \M^1(S\cup\{\theta_j\}) \\
			\mbox{}         &                 & \mbox{}         \\
			\M^{d-|S|}(S)[-1] & \M^{d-|S|}(S) & \M^{d-|S|}(S\cup\{\theta_j\}) & 0.\\
		};
		
		\path[overlay,->, font=\scriptsize,>=latex]
		(m-1-1) edge  node [descr, yshift=1.1ex] {} (m-1-2)
		(m-1-2) edge node [descr, yshift=1.1ex] {}(m-1-3)
		(m-1-3) edge[out=355,in=175] node[descr,yshift=0.3ex] {$\delta$} (m-2-1)
		(m-2-1) edge node [descr, yshift=1.1ex] {$\cdot\theta_j$} (m-2-2)
		(m-2-2) edge node [descr, yshift=1.1ex] {} (m-2-3)
		(m-2-3) edge[out=355,in=175,dashed] (m-4-1)
		(m-4-1) edge node [descr, yshift=1.15ex] {$\cdot\theta_j$} (m-4-2)
		(m-4-2) edge node [descr, yshift=1.1ex] {}(m-4-3)
		(m-4-3) edge node [descr, yshift=1.1ex] {}(m-4-4);
		\end{tikzpicture}
		\]
	
\end{lemma}

\begin{proof}
	As in the proof of Lemma \ref{primeavoidance}, denote
		\[
			\N(S):=\frac{\field[\Delta]/S}{\M^0(S)}.
		\]
	Once more, the Krull dimension of $\N(S)$ is positive since $S\subsetneq \Theta$. Furthermore, the depth of $\N(S)$ is positive because $H_\mideal^0(\N(S))=0$ and $\mideal\cdot \N(S)\not= \N(S)$, and so we may assume by the genericity of $\Theta$ that $\theta_j$ is a non-zero-divisor on $\N(S)$. Hence, there is a short exact sequence of the form
		\[
			0\to \N(S)[-1] \xrightarrow{\cdot\theta_j} \frac{\field[\Delta]/S}{\theta_j\cdot \left(\M^0(S)[-1]\right)} \to \field[\Delta]/(S\cup\{\theta_j\})\to 0,
		\]
	inducing a long exact sequence in local cohomology of the form
		\begin{equation}\label{bigLES}
		\begin{tikzpicture}[scale = 1, descr/.style={fill=white,inner sep=1.5pt}]
		\matrix (m) [
		matrix of math nodes,
		row sep=1.5em,
		column sep=1em,
		text height=1ex, text depth=0.25ex
		]
		{ H_\mideal^0(\N(S))[-1] & H_\mideal^0\left(\frac{\field[\Delta]/S}{\theta_j\cdot \left(\M^0(S)[-1]\right)}\right) & \M^0(S\cup\{\theta_j\}) \\
			 H_\mideal^1(\N(S))[-1] & H_\mideal^1\left(\frac{\field[\Delta]/S}{\theta_j\cdot \left(\M^0(S)[-1]\right)}\right) & \M^1(S\cup\{\theta_j\}) \\
			\mbox{}         &                 & \mbox{}         \\
		 H_\mideal^{d-|S|}(\N(S))[-1] & H_\mideal^{d-|S|}\left(\frac{\field[\Delta]/S}{\theta_j\cdot \left(\M^0(S)[-1]\right)}\right) & \M^{d-|S|}(S\cup\{\theta_j\}).\\
		};
		
		\path[overlay,->, font=\scriptsize,>=latex]
		(m-1-1) edge  node [descr, yshift=1.1ex] {$\cdot\theta_j$} (m-1-2)
		(m-1-2) edge node [descr, yshift=1.1ex] {}(m-1-3)
		(m-1-3) edge[out=355,in=175] node[descr,yshift=0.3ex] {$\delta$} (m-2-1)
		(m-2-1) edge node [descr, yshift=1.1ex] {$\cdot\theta_j$} (m-2-2)
		(m-2-2) edge node [descr, yshift=1.1ex] {} (m-2-3)
		(m-2-3) edge[out=355,in=175,dashed] (m-4-1)
		(m-4-1) edge node [descr, yshift=1.15ex] {$\cdot\theta_j$} (m-4-2)
		(m-4-2) edge node [descr, yshift=1.1ex] {}(m-4-3);
		\end{tikzpicture}
		\end{equation}
		Since $\M^0(S)\subset \field[\Delta]/S$ is an Artinian submodule, $H_\mideal^i(\N(S))\cong \M^i(S)$ for $i>0$. Thus, we can replace both $H_\mideal^i(\N(S))$ and $H_\mideal^i\left(\frac{\field[\Delta]/S}{\theta_j(\cdot \M^0(S)[-1])}\right)$ with $\M^i(S)$ for $i>0$ in this sequence. Furthermore, the first term, $H_\mideal^0(\N(S))$, is zero, and thus it only remains to show that the natural map
		\[
		\frac{\M^0(S)}{\theta_j\cdot (\M^0(S)[-1])}\to \M^0(S\cup\{\theta_j\})
		\]
		is an injection. For this step, consider the short exact sequence
		\[
		0 \to \theta_j\cdot (\M^0(S)[-1])\to \field[\Delta]/S\to \frac{\field[\Delta]/S}{\theta_j\cdot (\M^0(S)[-1])}\to 0
		\]
		where the maps involved are the inclusion and projection. Since $H_\mideal^1(\theta_j\cdot (\M^0(S)[-1]))=0$, the induced long exact sequence in local cohomology begins with
		\[
		0\to H_\mideal^0(\theta_j\cdot (\M^0(S)[-1]))\to \M^0(S)\to H_\mideal^0\left(\frac{\field[\Delta]/S}{\theta_j\cdot (\M^0(S)[-1])}\right)\to 0.
		\]
		However, $H_\mideal^0(\theta_j\cdot (\M^0(S)[-1]))=\theta_j\cdot (\M^0(S)[-1])$, and thus we obtain the natural isomorphism
		\[
		\frac{\M^0(S)}{\theta_j\cdot (\M^0(S)[-1])}\cong H_\mideal^0\left(\frac{\field[\Delta]/S}{\theta_j\cdot (\M^0(S)[-1])}\right).
		\]
	By replacing the appropriate term in the long exact sequence \eqref{bigLES}, the statement of the lemma follows.
\end{proof}

Unfortunately, we cannot use trivial module structure to directly split the long exact sequence above into a series of short exact ones as in the Buchsbaum case. However, we can still use the standard implied short exact sequences to allow for further analysis of the kernels that we need to study. These sequences are shown in the following lemma, which follows immediately from Lemma \ref{LESLemma} and the definitions.

\begin{lemma}\label{diagramlemma}
	Let $S\subset\Theta$ be such that $\theta_{j_1}, \theta_{j_2}\in \Theta\smallsetminus S$ and $\theta_{j_1}\not=\theta_{j_2}$. Then for any $i$, there is a commutative diagram with exact rows of the form
		\[
			\begin{tikzcd}
			0 \ar{r} & \C\K^i(S, \theta_{j_1}) \ar{r}\ar{d}{\cdot\theta_{j_2}} & \M^i(S\cup\{\theta_{j_1}\}) \ar{r}\ar{d}{\cdot\theta_{j_2}} & \K^{i+1}(S, \theta_{j_1})[-1] \ar{r}\ar{d}{\cdot\theta_{j_2}} & 0 \\
						0 \ar{r} & \C\K^i(S, \theta_{j_1})[1] \ar{r} & \M^i(S\cup\{\theta_{j_1}\})[1] \ar{r} & \K^{i+1}(S, \theta_{j_1}) \ar{r} & 0.
			\end{tikzcd}
		\]
\end{lemma}

The next proposition is the first example of reducing statements about $\M^i(j)$ for varying $i$ and $j$ to statements about $\M^{i+j}(\emptyset)$, whose structure is well-understood through Gr\"{a}be's theorem.

\begin{proposition}\label{submodProp}
	Let $\Delta$ be a $(d-1)$-dimensional simplicial complex and let $\theta_1, \ldots, \theta_d$ be a generic l.s.o.p.~ for $\field[\Delta]$. Then for all $i$ and $j$, there exists a submodule $\Bb^i(j)$ of $\K^i(j)$ that satisfies the following three conditions:
	\begin{enumerate}[label=(\roman*)]
		\item There is a natural isomorphism
			\[
				\left(\frac{\K^i(j)}{\Bb^i(j)}\right)_{\ge j-2}\cong \LL_{j}^{i+j-1}[1-j].
			\]
		\item $\Bb^i(j)$ is concentrated in degree $j-2$. In particular, $\Bb^i(j)$ has trivial module structure.
		\item The dimension of $\Bb^i(j)$ as a vector space over $\field$ in degree $j-2$ is
			\[
				\D\left(\Bb^i(j)_{j-2}\right)=\sum_{p=1}^{j-1}\D \left[\Coker\cdot\theta_p:\LL_{p-1}^{i+j-2}\to \M^{i+j-2}(\emptyset)_0\right].
			\]
	\end{enumerate}
\end{proposition}

\begin{proof}
	The proof will proceed by induction on $j$. For the inductive step, we will need to establish both base cases $j=1$ and $j=2$. For $j=1$, taking $\Bb^{i}(1)=\{0\}$ satisfies all three conditions of the proposition, as the sum in part (iii) is empty and $(\K^i(1))_{\ge -1}=\LL_{1}^i$.
	
	For the $j=2$ case, let $\Bb^i(2)=\C\K^i(1)_0$. This choice immediately satisfies condition (ii) since $\C\K^i(1)_0$ is concentrated in a single degree, and the equality
		\[
			\C\K^i(1)_0 = \Coker \cdot \theta_1:\LL_0^i\to \M^i_0
		\]
	demonstrates that it also satisfies condition (iii). For condition (i), the commutative diagram (with exact rows) guaranteed by Lemma \ref{diagramlemma} can be truncated and written in the following form:
		\[
			\begin{tikzcd}
			0 \ar{r} & \C\K^i(1)_0 \ar{r}\ar{d}{\cdot\theta_2} & \M^i(1)_{\ge 0} \ar{r}\ar{d}{\cdot\theta_2} & \LL^{i+1}_{1}[-1] \ar{r}\ar{d}{\cdot\theta_2} & 0 \\
			0 \ar{r} & 0 \ar{r} & \M^i(1)[1]_0 \ar{r} & \M^{i+1}(\emptyset)_0 \ar{r} & 0.
			\end{tikzcd}
		\]
	Applying the snake lemma to this diagram results in the short exact sequence
		\[
			0 \to \C\K^i(1)_0 \to \K^i(2)_{\ge 0} \to \Ker (\cdot \theta_2:\LL^{i+1}_{1}[-1]\to \M^{i+1}(\emptyset)_0)\to 0,
		\]
	so
		\[
			\left(\frac{\K^i(2)}{\Bb^i(2)}\right)_{\ge 0}=\left(\frac{\K^i(2)}{\C\K^i(1)}\right)_{\ge 0}\cong \Ker (\cdot \theta_2:\LL^{i+1}_{1}[-1]\to \M^{i+1}(\emptyset)_0).
		\]
	Finally, note that this rightmost term is nothing more than $\LL_{2}^{i+1}[-1]$, establishing the $j=2$ case.
	
	Now suppose that $j\ge 2$ and that $\K^i(\ell)$ contains a submodule $\Bb^i(\ell)$ satisfying the properties of the proposition for all $i$ and all $\ell\le j$. As in the $j=2$ case, the main tool will be the truncated version of the diagram from Lemma \ref{diagramlemma}, with exact rows:
		\[
			\begin{tikzcd}
			0 \ar{r} & \C\K^i(j)_{j-1} \ar{r}\ar{d}{\cdot\theta_{j+1}} & \M^i(j)_{\ge j-1} \ar{r}\ar{d}{\cdot\theta_{j+1}} & \K^{i+1}(j)[-1]_{\ge j-1} \ar{r}\ar{d}{\cdot\theta_{j+1}} & 0 \\
			0 \ar{r} & 0 \ar{r} & \M^i(j)[1]_{j-1} \ar{r} & \M^{i+1}(j-1)_{j-1} \ar{r} & 0.
			\end{tikzcd}
		\]
	Once more, the snake lemma implies the existence of a natural isomorphism
		\begin{equation}\label{varphieq}
			\varphi:\left(\frac{\K^i(j+1)}{\C\K^i(j)}\right)_{\ge j-1}\xrightarrow{\sim} \Ker \left(\cdot \theta_{j+1}:\K^{i+1}(j)[-1]_{\ge j-1}\to \M^{i+1}(j-1)_{j-1}\right).
		\end{equation}
	By the inductive hypothesis, there exists a submodule $\Bb^{i+1}(j)\subset \K^{i+1}(j)_{\ge j-2}$ satisfying the properties of the proposition. In particular,
		\[
			\left(\frac{\K^{i+1}(j)}{\Bb^{i+1}(j)}[-1]\right)_{\ge j-1}\cong \LL_{j}^{i+j}[1-(j+1)].
		\]
	Now define
		\[
			\Bb^i(j+1):=\C\K^i(j)_{j-1}+\varphi^{-1}(\Bb^{i+1}(j)[-1]).
		\]
	It is immediate that $\Bb^i(j+1)$ satisfies property (ii) of the proposition. Furthermore, taking the quotient of both sides of \eqref{varphieq} produces an isomorphism
		\[
			\left(\frac{\K^i(j+1)}{\Bb^{i}(j+1)}\right)_{\ge (j+1)-2}\cong \Ker (\cdot \theta_{j+1}:\LL_{j}^{i+j}[1-(j+1)]\to \M^{i+j}(\emptyset)[1-(j+1)]).
		\]
	Since this kernel is just $\LL_{j+1}^{i+(j+1)-1}[1-(j+1)]$, property (i) of the proposition has been established. It remains to calculate $\D\left( \Bb^{i}(j+1)_{j-1}\right)$ and establish property (iii).
	
	By assumption, we immediately have
		\begin{equation}\label{Mdim}
			\D\left(\varphi^{-1}(\Bb^{i+1}(j)[-1])_{j-1}\right)=\sum_{p=1}^{j-1}\D \left[\Coker\cdot\theta_p:\LL_{p-1}^{i+(j+1)-2}\to \M^{i+(j+1)-2}(\emptyset)_0\right].
		\end{equation}
	Thus, it remains to calculate the dimension of $\C\K^i(j)_{j-1}$. For this, once again consider the commutative diagram with exact rows
		\[
			\begin{tikzcd}
			0 \ar{r} & \C\K^i(j-1)_{j-2} \ar{r}\ar{d}{\cdot\theta_{j}} & \M^i(j-1)_{\ge j-2} \ar{r}\ar{d}{\cdot\theta_{j}} & \K^{i+1}(j-1)[-1]_{\ge j-2} \ar{r}\ar{d}{\cdot\theta_{j}} & 0 \\
			0 \ar{r} & 0 \ar{r} & \M^i(j-1)[1]_{j-2} \ar{r} & \M^{i+1}(j-2)_{j-2} \ar{r} & 0.
			\end{tikzcd}
		\]
	In this case, the snake lemma provides an immediate isomorphism
		\begin{equation}\label{cokereq}
			\C\K^i(j)_{j-1}\cong (\Coker \cdot\theta_j: \K^{i+1}(j-1)[-1]_{\ge j-2}\to \M^{i+1}(j-2)_{j-2}).
		\end{equation}
	Applying our inductive hypothesis once more provides a submodule $\Bb^{i+1}(j-1)\subset \K^{i+1}(j-1)$ such that
		\[
			\left(\frac{\K^{i+1}(j-1)}{\Bb^{i+1}(j-1)}\right)_{\ge j-3}\cong \LL_{j-1}^{i+j-1}[2-j].
		\]
	Furthermore, since $\Bb^{i+1}(j-1)$ has trivial module structure, the map $\cdot\theta_j$ descends to the quotient and we obtain the next commutative diagram in which the horizontal maps are isomorphisms:
		\[
		\begin{tikzcd}
			\left(\frac{\K^{i+1}(j-1)}{\Bb^{i+1}(j-1)}\right)[-1]_{\ge j-2} \ar{r}\ar{d}{\cdot\theta_j} & \LL_{j-1}^{i+j-1}[-(j-1)] \ar{d}{\cdot\theta_j} \\
			\M^{i+1}(j-2))_{j-2}\ar{r} & \M^{i+j-1}(\emptyset)[-(j-2)]_{j-2}.
		\end{tikzcd}
		\]
	
	Finally, the cokernel we are interested in from equation \eqref{cokereq} is unaffected by taking the quotient. That is, $\C\K^i(j)_{j-1}$ is isomorphic to the cokernel of the left vertical map above, and thus the diagram implies that
		\[
			\C\K^i(j)_{j-1}\cong (\Coker \cdot\theta_j: \LL_{j-1}^{i+(j+1)-2}[-(j-1)] \to \M^{i+(j+1)-2}(\emptyset)[-(j-2)]).
		\]
	This, combined with \eqref{Mdim}, shows that $\Bb^i(j+1)$ also has the appropriate dimension (property (iii)), completing the proof.
\end{proof}

The calculations in the above proposition allow for us to further rephrase properties of successive quotients in our Hilbert series calculations back to statements about submodules of $\M^j(\emptyset)$. At this point, we can revert these dimensions back to the relevant Betti numbers of links of certain faces.

\begin{corollary}\label{kernelDim}
	If $\Delta$ is a $(d-1)$-dimensional simplicial complex, then
	\[
	\D\left(\K^i(j)_{j-2}\right) = (j-1)\widetilde{\beta}_{i+j-3}(\Delta)+\D\left((\LL_{j}^{i+j-1})_{-1}\right)+ \D\left((\LL_{j-1}^{i+j-2})_{-1}\right)-\sum_{v\in V}\beta_{i+j-3}(\Delta, \cost_\Delta v).
	\]
	for all $i$ and $j=1, \ldots, d$.
\end{corollary}

\begin{proof}
	Since
		\[
			\LL_{j}^i = \Ker \cdot \theta_j:\LL_{j-1}^i\to \M^i(\emptyset)_0,
		\]
	we can write
		\begin{align*}
			\D \left[\Coker \cdot\theta_p:\LL_{p-1}^{i+j-2}\to \M^{i+j-2}(\emptyset)_0\right] &=
			\D (\M^{i+j-2}(\emptyset)_0)-\D\left[\im \cdot\theta_p:\LL_{p-1}^{i+j-2}\to \M^{i+j-2}(\emptyset)_0\right]\\
			&=\D (\M^{i+j-2}(\emptyset)_0) - \left[\D\left((\LL_{p-1}^{i+j-2})_{-1}\right)-\D\left((\LL_{p}^{i+j-2})_{-1}\right)\right]\\
			&= \widetilde{\beta}_{i+j-3}(\Delta) + \D\left((\LL_{p}^{i+j-2})_{-1}\right) - \D\left((\LL_{p-1}^{i+j-2})_{-1}\right).
		\end{align*}
	Hence,
		\begin{align*}
			\sum_{p=1}^{j-1} \D\left[\Coker\cdot\theta_p:\LL_{p-1}^{i+j-2}\to \M^{i+j-2}(\emptyset)_0\right]&= 
			\sum_{p=1}^{j-1} \left(\widetilde{\beta}_{i+j-3}(\Delta) + \D\left((\LL_{p}^{i+j-2})_{-1}\right) - \D\left((\LL_{p-1}^{i+j-2})_{-1}\right)\right)\\
			&=(j-1)\widetilde{\beta}_{i+j-3}(\Delta)+\D\left((\LL_{j-1}^{i+j-2})_{-1}\right)-\D\left((\LL_{\emptyset}^{i+j-2})_{-1})\right)\\
			&=(j-1)\widetilde{\beta}_{i+j-3}(\Delta)+\D\left((\LL_{j-1}^{i+j-2})_{-1}\right)-\D\left(\M^{i+j-2}(\emptyset)_{-1}\right)\\
			&=(j-1)\widetilde{\beta}_{i+j-3}(\Delta)+\D\left((\LL_{j-1}^{i+j-2})_{-1}\right)-\sum_{v\in V}\beta_{i+j-3}(\Delta, \cost_\Delta v).
		\end{align*}
	
	Combining the above equation with Proposition \ref{submodProp} and the equation
		\[
			\D\left(\K^i(j)_{j-2}\right) = 	\D\left( \Bb^i(j)_{j-2}\right)+	\D\left(\left(\frac{\K^i(j)}{\Bb^i(j)}\right)_{j-2}\right)
		\]
	completes the proof.
\end{proof}

\begin{theorem}\label{hd1Thm}
	If $\Delta$ is a $(d-1)$-dimensional simplicial complex, then
		\[
			h_{d-1}^\alg(\Delta)-h_{d-1}(\Delta)=\D\left((\LL_{d}^{d-1})_{-1}\right)+(-1)^{d-1}\sum_{F\in \Delta}\binom{d-|F|}{d-1}\widetilde{\chi}_{d-3-|F|}(\lk_\Delta F).
		\]
\end{theorem}

\begin{proof}
	By Lemma \ref{HilbertSeries}, $h_{d-1}^\alg(\Delta)-h_{d-1}(\Delta)$ is the coefficient on $t^{d-1}$ of the polynomial
		\[
			\sum_{j=1}^d(1-t)^{d-j}t\Hilb(\K^0(j), t),
		\]
	so that
		\[
			h_{d-1}^\alg(\Delta)-h_{d-1}(\Delta)=\sum_{j=1}^d\left[\binom{d-j}{d-j-1}(-1)^{d-j-1}\D\left(\K^0(j)_{j-1}\right)+(-1)^{d-j}\D\left(\K^0(j)_{j-2}\right)\right].
		\]
	We will break this into two sums. First, since 
		\[
			\K^0(j)_{j-1}=\M^0(j-1)_{j-1}\cong \M^{j-1}(\emptyset)_0\cong \tilde{H}^{j-2}(\Delta),
		\]
	we obtain
		\[
			\sum_{j=1}^d\binom{d-j}{d-j-1}(-1)^{d-j-1}\D\left(\K^0(j)_{j-1}\right)=\sum_{j=0}^{d-2}(d-j-1)(-1)^{d-j}\widetilde{\beta}_{j-1}(\Delta).
		\]
	On the other hand, by Corollary \ref{kernelDim},
		\[
			\D\left(\K^0(j)_{j-2}\right) = (j-1)\widetilde{\beta}_{j-3}(\Delta)+\D\left((\LL_{j}^{j-1})_{-1}\right)+ \D\left((\LL_{j-1}^{j-2})_{-1}\right)-\sum_{v\in V}\beta_{j-3}(\Delta, \cost_\Delta v).
		\]

	Thus,
		\begin{align*}
			\sum_{j=1}^d (-1)^{d-j}&\D\left(\K^0(j)_{j-2}\right)\\ &=\sum_{j=1}^d(-1)^{d-j}\left[(j-1)\widetilde{\beta}_{j-3}(\Delta)+\D\left((\LL_{j}^{j-1})_{-1}\right)+ \D\left((\LL_{j-1}^{j-2})_{-1}\right)-\sum_{v\in V}\beta_{j-3}(\Delta, \cost_\Delta v)\right]\\
			&=\D\left((\LL_{d}^{d-1})_{-1}\right)+\sum_{j=0}^{d-2}(-1)^{d-j}\left[(j+1)\widetilde{\beta}_{j-1}(\Delta)-\sum_{v\in V}\beta_{j-1}(\Delta, \cost_\Delta v)\right].
		\end{align*}
	Combining the two parts together shows that
		\[
			h_{d-1}^\alg(\Delta)-h_{d-1}(\Delta)=\D\left((\LL_{d}^{d-1})_{-1}\right)+\sum_{j=0}^{d-2}(-1)^{d-j}\left[d\widetilde{\beta}_{j-1}(\Delta)-\sum_{v\in V}\beta_{j-1}(\Delta, \cost_\Delta v)\right],
		\]
		and re-writing this expression using the truncated Euler characteristic and equation (\ref{costtolink}) results in the statement of the theorem.
\end{proof}

\subsection{Sigma submodules}

We will begin this subsection with two lemmas detailing the content and structure of the sigma submodule. First, given $S\subsetneq\Theta$ and $\theta_j\not\in S$, let $\varphi_j$ and $\pi_j$ be the connecting homomorphism and the map induced by the projection, respectively, in the long exact sequence
	\[
	\cdots \to \M^i(S)\xrightarrow{\cdot\theta_j}\M^i(S)\xrightarrow{\pi_j}\M^i(S\cup\{\theta_j\}))\xrightarrow{\varphi_j}\M^{i+1}(S)\to\cdots
	\]
provided by Lemma \ref{LESLemma}. We also must introduce two final abbreviations
\[
\hat{\Theta}_i := (\theta_1, \ldots, \theta_{i-1}, \theta_{i+1}, \ldots, \theta_d)
\]
and
\[
\hat{\Theta}_{i, j} := (\theta_1, \ldots, \theta_{i-1}, \theta_{i+1}, \ldots, \theta_{j-1}, \theta_{j+1}, \ldots, \theta_d).
\]

\begin{lemma}\label{lem1}
	If $\Theta$ is a generic linear system of parameters for $\field[\Delta]$ and $m\in\Sigma(\Theta; \field[\Delta])/\Theta$, then $m\in \pi_i(\M^0(\hat{\Theta}_i))$ for some $i$.
\end{lemma}

\begin{proof}
	Assume that $m\not=0$. First note that $\M^0(\Theta)=\field[\Delta]/\Theta$, since $\field[\Delta]/\Theta$ is finite-dimensional. By assumption, $m$ satisfies $m\in (\hat{\Theta}_i):\theta_i$ for some $i$, while $m\not\in(\hat{\Theta}_i)$. Hence,
	there exists $n\in \field[\Delta]/\hat{\Theta}_i$ such that $\pi_i(n)=m$. Furthermore,
	\[
	n\in\Ker(\cdot\theta_i:\field[\Delta]/\hat{\Theta}_i\to \field[\Delta]/\hat{\Theta}_i).
	\]
	But by Lemma \ref{primeavoidance}, this means that $n\in \K^0(\hat{\Theta}_i, \theta_i)$. In particular, $n\in \M^0(\hat{\Theta}_i)$, and hence $m\in \pi_i(\M^0(\hat{\Theta}_i))$.
\end{proof}

\begin{lemma}\label{lem2}
	If $\Theta=(\theta_1, \ldots, \theta_d)$ is a generic linear system of parameters for $\field[\Delta]$, then
	\[
	\pi_i(\M^0(\hat{\Theta}_i))_{d-1}\cap 	\pi_j(\M^0(\hat{\Theta}_j))_{d-1}=\{0\}
	\]
	for $i\not=j$.
\end{lemma}

\begin{proof}
	Suppose that $m\in\M^0(\hat{\Theta}_i)$ and $n\in\M^0(\hat{\Theta}_j)$ satisfy $\pi_i(m)=\pi_j(n)$, and consider the diagram
	\[
	\begin{tikzcd}
	\M^0(\hat{\Theta}_j)_{d-1} \ar{dd}{\varphi_i} \ar{rr}{\pi_j} & & \M^0(\Theta)_{d-1} \ar{dd}{\varphi_i} \\
	\\
	\M^1(\hat{\Theta}_{i, j})_{d-2} \ar{rr}{\pi_j} & & \M^1(\hat{\Theta}_i)_{d-2},
	\end{tikzcd}
	\]
	which commutes by \cite[Appendix]{MNY}. Since $(\varphi_i\circ\pi_i)(m)=0$ and $\pi_i(m)=\pi_j(n)$, it must be that $(\varphi_i\circ\pi_j)(n)=0$. Then by commutativity, $(\pi_j\circ\varphi_i)(n)=0$. However, $\varphi_i:\M^0(\hat{\Theta}_j)\to \M^1(\hat{\Theta}_{i, j})[-1]$ is an isomorphism in degree $d-1$, and hence $\varphi_i(n)\not=0$. Since
	\[
	\Ker \left[\pi_j: \M^1(\hat{\Theta}_{i, j})_{d-2} \to \M^1(\hat{\Theta}_i)_{d-2} \right] = \im \left[\cdot \theta_j: \M^1(\hat{\Theta}_{i, j})_{d-3}\to \M^1(\hat{\Theta}_{i, j})_{d-2}\right],
	\]
	this implies that $\varphi_i(n) = \theta_j\cdot q $ for some $q\in \M^1(\hat{\Theta}_{i, j})_{d-3}$. That is, $\theta_j\cdot q$ is mapped to zero under the projection
	\[
	\M^1(\hat{\Theta}_{i, j})\to \C\K^1(\hat{\Theta}_{i, j}, \theta_j).
	\]
	However, $\C\K^1(\hat{\Theta}_{i, j}, \theta_j)_{d-2}$ is naturally isomorphic to $\C\K^0(\hat{\Theta}_{j}, \theta_j)_{d-1}$ via the descent of $\varphi_i$. Hence, $n=0$ in $\C\K^0(\hat{\Theta}_{j}, \theta_j)_{d-1}$ as well. Since $\pi_j$ factors through the projection
	\[
	\M^0(\hat{\Theta}_{j})_{d-1}\to \C\K^0(\hat{\Theta}_{j}, \theta_j)_{d-1},
	\]
	we have that $\pi_j(n)=0$.
	
\end{proof}

With these lemmas in place, we can now move on to proving the main result.

\begin{theorem}\label{sigmamodThm}
	If $\Delta$ is a $(d-1)$-dimensional simplicial complex, then
	\[
	h_{d-1}^\sig(\Delta) = h_{d-1}(\Delta) + (-1)^{d-1}\sum_{F\in \Delta}\binom{d-|F|}{d-1}\widetilde{\chi}_{d-2-|F|}(\lk_\Delta F).
	\]
\end{theorem}

\begin{proof}
	First note that since $\mideal\cdot \M^0(\hat{\Theta}_{i})_{d-1}=0$, it must be that $\pi_i(\M^0(\hat{\Theta}_{i}))_{d-1}\subseteq \left(\frac{\Sigma(\Theta; \field[\Delta])}{\Theta}\right)_{d-1}$ for all $i$. Combining this with Lemmas \ref{lem1} and \ref{lem2},
	\[
	\left(\frac{\Sigma(\Theta; \field[\Delta])}{\Theta}\right)_{d-1}\cong \bigoplus_{i=1}^d \pi_i(\M^0(\hat{\Theta}_{i}))_{d-1}.
	\]
	Moreover, as noted in the previous proof, $\pi_i$ factors through the projection to produce an injection
	\[
	\M^0(\hat{\Theta}_{i})_{d-1}\to \C\K^0(\hat{\Theta}_{i}, \theta_i)_{d-1}\hookrightarrow \M^0(\Theta)_{d-1}.
	\]
	Now combining this with the dimension calculation of $\C\K^0(\hat{\Theta}_{i}, \theta_i)_{d-1}$ in Proposition \ref{submodProp} along with the generic assumption on $\Theta$, we obtain
	\[
	\D\left(\frac{\Sigma(\Theta; \field[\Delta])}{\Theta}\right)_{d-1} = d \cdot \D\left[\Coker \cdot \theta_i: \LL_d^{d-1}\to \M^{d-1}(\emptyset)\right]
	\]
	for any $i$. Once more, since $\LL_d^{d-1}=\Ker\left[\cdot\theta_d: \LL_{d-1}^{d-1}\to \M^{d-1}(\emptyset)\right]$, this can be re-written as
	\begin{align*}
		\D\left(\frac{\Sigma(\Theta; \field[\Delta])}{\Theta}\right)_{d-1} & = d\cdot \left[\widetilde{\beta}_{d-2}(\Delta)+ \D\left((\LL_{d}^{d-1})_{-1}\right)-\D\left((\LL_{d-1}^{d-1})_{-1}\right) \right].
	\end{align*}
	Now consider the maps
		\[
			f_i:\M^{d-1}(\emptyset)_{-1}\to \bigoplus_{j=1}^i \M^{d-1}(\emptyset)_0
		\]
	defined componentwise by $f_i(m)=(\theta_1\cdot m, \theta_2\cdot m, \ldots, \theta_i\cdot m)$. Then $\Ker f_i = (\LL_i^{d-1})_{-1}$, so we obtain
		\begin{equation}\label{ld}
			\D\left((\LL_{i}^{d-1})_{-1}\right)=\D(\Ker f_i) =  \D\left(\M^{d-1}(\emptyset)_{-1}\right)-\sum_{j=1}^i\D\left[\im\cdot \theta_j:\M^{d-1}(\emptyset)_{-1}\to  \M^{d-1}(\emptyset)_0\right].
		\end{equation}
	In particular,
		\[
			\D\left((\LL_{d}^{d-1})_{-1}\right)-\D\left((\LL_{d-1}^{d-1})_{-1}\right) = \D\left[\im\cdot \theta_d:\M^{d-1}(\emptyset)_{-1}\to  \M^{d-1}(\emptyset)_0\right],
		\]
	so 
		\[
			\D\left(\frac{\Sigma(\Theta; \field[\Delta])}{\Theta}\right)_{d-1} = d\cdot \widetilde{\beta}_{d-2}(\Delta)+d\cdot  \D\left[\im\cdot \theta_d:\M^{d-1}(\emptyset)_{-1}\to  \M^{d-1}(\emptyset)_0\right].
		\]
	Combining this equality with Theorem \ref{hd1Thm} yields
		\begin{align*}
			\D\left(\frac{\field[\Delta]}{\Sigma(\Theta; \field[\Delta])}\right)_{d-1} & = \D\left(\frac{\field[\Delta]}{\Theta}\right)_{d-1}-\D\left(\frac{\Sigma(\Theta; \field[\Delta])}{\Theta}\right)_{d-1}\\
			&= h_{d-1}(\Delta) + \D\left((\LL_{d}^{d-1})_{-1}\right) + (-1)^{d-1}\sum_{F\in \Delta}\binom{d-|F|}{d-1}\widetilde{\chi}_{d-3-|F|}(\lk_\Delta F)\\
			&\hspace{20pt}- \left(d\cdot \widetilde{\beta}_{d-2}(\Delta)+d\cdot  \D\left[\im\cdot \theta_d:\M^{d-1}(\emptyset)_{-1}\to  \M^{d-1}(\emptyset)_0\right]\right).
		\end{align*}
	Now using \eqref{ld} and appealing once more to the genericity of $\Theta$, from which it follows that
		\[
			\D\left[\im\cdot \theta_d:\M^{d-1}(\emptyset)_{-1}\to  \M^{d-1}(\emptyset)_0\right]=\D\left[\im\cdot \theta_i:\M^{d-1}(\emptyset)_{-1}\to  \M^{d-1}(\emptyset)_0\right]
		\]
	for all $i$, we obtain
		\begin{align*}
		\D\left(\frac{\field[\Delta]}{\Sigma(\Theta; \field[\Delta])}\right)_{d-1} & = h_{d-1}(\Delta)  + (-1)^{d-1}\sum_{F\in \Delta}\binom{d-|F|}{d-1}\widetilde{\chi}_{d-3-|F|}(\lk_\Delta F)\\
		& \hspace{20pt} + \D\left(\M^{d-1}(\emptyset)_{-1}\right) - d\cdot \widetilde{\beta}_{d-2}(\Delta) \\
		&= h_{d-1}(\Delta)+(-1)^{d-1}\sum_{F\in \Delta}\binom{d-|F|}{d-1}\widetilde{\chi}_{d-2-|F|}(\lk_\Delta F).
		\end{align*}
\end{proof}

\section{An application to suspensions of Buchsbaum complexes}\label{sect:apps}

A $(d-1)$-dimensional complex $\Delta$ is said to have isolated singularities if $\widetilde{\beta}_i(\lk_\Delta F)=0$ for all $i<d-1-|F|$ and all faces $F\in\Delta$ with $|F|\ge 2$ while failing this condition for at least one face $F$ with $|F|=1$. As stated in Section \ref{sect:intro}, algebraic $h$-vectors of complexes with isolated singularities have been studied in some depth in \cite{MNS-sing} and \cite{NS-sing}. However, those computations depended upon a further assumption that the singularities of $\Delta$ are homologically isolated, stipulating that images of inclusion maps of the form
	\[
	H^i(\Delta, \cost_\Delta v)\to H^i(\Delta, \cost_\Delta \emptyset)
	\]
across singular vertices $v$ have trivial intersection for $i<d-1$. It turns out that in this case, \cite[Lemma 4.3]{NS-sing} shows that quotienting $\field[\Delta]$ only by $(\theta_1)$ results in a Buchsbaum $A$-module, allowing for an easy resumption of Schenzel's classic techniques on further quotients to compute the algebraic $h$-vector (in fact, this is an equivalent characterization of homological isolation of singularities, see \cite[Proposition 4.5]{AlmostBuchs}).

When singularities are not homologically isolated, the situation has the potential to become much more complex. In particular, even the simple case of $\Delta$ being a triangulation of the suspension of a manifold that is not a sphere, not much has previously been written about $h^\alg(\Delta)$. In this section we will use the methods of Section \ref{sect:results} to calculate the $h^\alg$-vector for such a complex.

Let $\Delta$ be a triangulation of the suspension of (the geometric realization of) a Buchsbaum complex that is not Cohen--Macaulay. It turns out to be relatively straightforward to calculate the structure and dimensions of $\M^i(j)$ for arbitrary $j$, as \cite[Section 3.2]{AlmostBuchs} shows that
	\begin{equation}\label{suspDims}
		\M^i(2)_j\cong \left\{\begin{array}{cc} \widetilde{H}^{i}(\Delta) & j=0 \\ \widetilde{H}^{i+1}(\Delta) & j=2. \end{array}\right.
	\end{equation}
Given that $\M^i(2)$ is concentrated in degrees $0$ and $2$, it has trivial module structure. Hence, the same is seen to be true of $\M^i(j)$ for all $j\ge 2$ by recursively applying Lemma \ref{diagramlemma}, and
	\[
		\M^i(j+1)\cong \M^i(j)\bigoplus\left( \M^{i+1}(j)[-1]\right)
	\]
for $j\ge 2$. Then
	\[
		\M^0(j)\cong \bigoplus_{i=0}^{j-2}\left(\bigoplus_{\binom{j-2}{i}}\M^i(2)[-i]\right),
	\]
and in consideration of \eqref{suspDims},
	\begin{equation}
		\D\left(\M^0(j)_i\right)=\binom{j-2}{i}\widetilde{\beta}_i(\Delta)+\binom{j-2}{i-2}\widetilde{\beta}_{i-1}(\Delta)
	\end{equation}
for $j\ge 2$. The technique of Section \ref{sect:results} can now be applied to calculate the full $h^\alg$-vector for a complex of this type.

\begin{theorem}\label{SuspThm}
	Let $\Delta$ be a triangulation of the suspension of (the geometric realization of) a Buchsbaum complex, and suppose that $\Delta$ is of dimension $d-1$. Then
		\[
			h_{i}^\alg(\Delta)-h_{i}(\Delta)=(-1)^i\left[\binom{d-2}{i-2}\widetilde{\chi}_{i-2}(\Delta) - \binom{d-2}{i}\widetilde{\chi}_{i-1}(\Delta)\right]
		\]
		for $i=0, \ldots, d$.
\end{theorem}

\begin{proof}
Since $\M^i(j)$ has trivial module structure for $j\ge 2$ from the above comments,
	\[
		\K^0(j) =\M^0(j-1)
	\]
for $j\ge 3$. Also, since $\Delta$ triangulates a suspension, it is connected. Thus, the depth of $\field[\Delta]$ is at least $2$ (see \cite[Corollary 2.6]{HibiDepth}), so $\K^0(1)=\K^0(2)=0$ and, by Lemma \ref{HilbertSeries},
	\begin{align*}
		\Hilb(\field[\Delta]/\Theta, t) &= \sum_{i=0}^d h_i(\Delta)t^i + \sum_{j=1}^d (1-t)^{d-j}t\Hilb(\K^0(j), t)\\
		&= \sum_{i=0}^d h_i(\Delta)t^i+\sum_{j=1}^d\left[(1-t)^{d-j}t\sum_{k}t^k\left(\binom{j-3}{k}\widetilde{\beta}_k(\Delta)+\binom{j-3}{k-2}\widetilde{\beta}_{k-1}(\Delta)\right)\right].
	\end{align*}
Hence, $h_{i+1}^\alg(\Delta)-h_{i+1}(\Delta)$ is equal to the coefficient on $t^{i+1}$ of the polynomial formed in the second summation above. Equivalently, considering the factor of $t$ in all terms of this sum, this $t^{i+1}$ coefficient is the same as the coefficient on $t^i$ of the polynomial
	\[
		\sum_{j=1}^d\left[(1-t)^{d-j}\sum_{k}t^k\left(\binom{j-3}{k}\widetilde{\beta}_k(\Delta)+\binom{j-3}{k-2}\widetilde{\beta}_{k-1}(\Delta)\right)\right],
	\]
which is the same as
	\begin{align*}
		&\sum_{j=1}^d\left[\sum_{k+\ell=i}(-1)^\ell\binom{d-j}{\ell}\left(\binom{j-3}{k}\widetilde{\beta}_k(\Delta)+\binom{j-3}{k-2}\widetilde{\beta}_{k-1}(\Delta)\right)\right]\\
		&=\sum_{j=1}^d\left[\sum_{k=0}^i(-1)^{i-k}\binom{d-j}{i-k}\left(\binom{j-3}{k}\widetilde{\beta}_k(\Delta)+\binom{j-3}{k-2}\widetilde{\beta}_{k-1}(\Delta)\right)\right].
	\end{align*}
This can now be broken into two sums. For the ``$\widetilde{\beta}_k(\Delta)$'' terms,
	\begin{align*}
		\sum_{j=1}^d\left[\sum_{k=0}^i(-1)^{i-k}\binom{d-j}{i-k}\binom{j-3}{k}\widetilde{\beta}_k(\Delta)\right]&=\sum_{k=0}^i(-1)^{i-k}\widetilde{\beta}_k(\Delta)\left[\sum_{j=1}^d\binom{d-j}{i-k}\binom{j-3}{k}\right]\\
		&=\sum_{k=0}^i(-1)^{i-k}\widetilde{\beta}_k(\Delta)\binom{d-2}{i+1}.
	\end{align*}
Similarly, the ``$\widetilde{\beta}_{k-1}(\Delta)$'' terms can be written as
	\begin{align*}
		\sum_{j=1}^d\left[\sum_{k=0}^i(-1)^{i-k}\binom{d-j}{i-k}\binom{j-3}{k-2}\widetilde{\beta}_{k-1}(\Delta)\right]&=\sum_{k=0}^i(-1)^{i-k}\widetilde{\beta}_{k-1}(\Delta)\left[\sum_{j=1}^d\binom{d-j}{i-k}\binom{j-3}{k-2}\right]\\
		&=\sum_{k=0}^i(-1)^{i-k}\widetilde{\beta}_{k-1}(\Delta)\binom{d-2}{i-1}.
	\end{align*}

Thus,
	\[
		h_{i+1}^\alg(\Delta)-h_{i+1}(\Delta)=(-1)^{i+1}\binom{d-2}{i-1}\widetilde{\chi}_{i-2}(\Delta)+(-1)^i\binom{d-2}{i+1}\widetilde{\chi}_i(\Delta).
	\]
\end{proof}

A straightforward calculation shows that if $\Delta$ is the direct suspension of a $(d-2)$-dimensional complex $\Gamma$, then 
	\begin{equation}\label{sushvector}
		h_i(\Delta)=h_i(\Gamma)+h_{i-1}(\Gamma)
	\end{equation}
	for $i=0, \ldots, d$. With this combinatorial relationship in mind, it is worth examining which relationships may exist between between the $h^\alg$-vectors of a Buchsbaum complex and its suspension now that Theorem \ref{SuspThm} can produce the $h^\alg$-vector of the suspension, resulting in the next corollary.

\begin{corollary}
	Let $\Gamma$ be a $(d-2)$-dimensional Buchsbaum complex, and let $\Delta$ be the suspension of $\Gamma$. Then
	\[
		h_i^\alg(\Delta)=h_i^\alg(\Gamma)+h_{i-1}^\alg(\Gamma)-\binom{d-2}{i-1}\widetilde{\beta}_{i-2}(\Gamma)
	\]
	for $i=0, \ldots, d$.
\end{corollary}

\begin{proof}
	Since $\widetilde{\beta}_j(\Delta)=\widetilde{\beta}_{j-1}(\Gamma)$ for all $j$, a shift in index on one of the sums in Theorem \ref{SuspThm} provides the equation
		\[
			h_i^\alg(\Delta)=h_i(\Delta)+ \binom{d-2}{i-2}\sum_{j=0}^{i-2}(-1)^{i-j}\widetilde{\beta}_{j-1}(\Gamma) + \binom{d-2}{i}\sum_{j=0}^{i-1}(-1)^{i-j-1}\widetilde{\beta}_{j-1}(\Gamma).
		\]
	On the other hand, Schenzel's Theorem \ref{Schenzel} tells us that
		\[
			h_i^\alg(\Gamma)=h_i(\Gamma)+\binom{d-1}{i}\sum_{j=0}^{i-1}(-1)^{i-j-1}\widetilde{\beta}_{j-1}(\Gamma).
		\]
	Then using \eqref{sushvector},
		\begin{align*}
			&\left[h_i^\alg(\Gamma) +h_{i-1}^\alg(\Gamma)\right] - h_i^\alg(\Delta)\\
			&=\left[\binom{d-1}{i}-\binom{d-2}{i}\right]\sum_{j=0}^{i-1}(-1)^{i-j-1}\widetilde{\beta}_{j-1}(\Gamma) + \left[\binom{d-1}{i-1}-\binom{d-2}{i-2}\right]\sum_{j=0}^{i-2}(-1)^{i-j}\widetilde{\beta}_{j-1}(\Gamma)\\
			&=\binom{d-2}{i-1}\sum_{j=0}^{i-1}(-1)^{i-j-1}\widetilde{\beta}_{j-1}(\Gamma)+\binom{d-2}{i-1}\sum_{j=0}^{i-2}(-1)^{i-j}\widetilde{\beta}_{j-1}(\Gamma)\\
			&= \binom{d-2}{i-1}\widetilde{\beta}_{i-2}(\Gamma)
		\end{align*}
\end{proof}

\section{Further comments and possible extensions}\label{sect:conclusion}

As the optimistic title of this paper suggests, we are hopeful that the current results can be further generalized and utilized to great effect. There are two principal avenues which we wish to investigate in the future.

\subsection{Lower entries in $h^\sig(\Delta)$}
Thus far we have only dealt with the specific case of $h_{d-1}^\sig(\Delta)$ afforded by the relatively nice description of local cohomology modules in special degrees stated in Proposition \ref{submodProp}. The fact that so many of the terms involved ``collapsed'' in such a spectacular fashion during the calculations in the proofs of Theorems \ref{hd1Thm} and \ref{sigmamodThm} has allowed for us to be very hopeful that these methods may be pushed further to accommodate a description of $h_i^\sig(\Delta)$ for \textit{all} values of $i$. In particular, a natural extension of Theorem \ref{sigmamodThm} provides the following conjecture.
\begin{conjecture}\label{mainconjecture}
	If $\Delta$ is a $(d-1)$-dimensional simplicial complex and $\Theta=\theta_1, \ldots, \theta_d$ is a generic l.s.o.p~ for $\field[\Delta]$, then there exists a submodule $\tau(\Theta; \field[\Delta])\subset \field[\Delta]$ such that
	\[
		\dim_\field\left(\frac{\field[\Delta]}{\tau(\Theta; \field[\Delta])} \right)_i = h_i(\Delta)+(-1)^i\sum_{F\in \Delta}\binom{d-|F|}{i}\widetilde{\chi}_{i-1-|F|}(\lk_\Delta F).
	\]
	Furthermore, in the case that $\Delta$ is Buchsbaum, $\tau(\Theta; \field[\Delta])=\Sigma(\Theta; \field[\Delta])$.
\end{conjecture}

An obvious candidate for the submodule $\tau(\Theta; \field[\Delta])$ in the above conjecture would be $\Sigma(\Theta;\field[\Delta])$, but it is very possible that a different choice may turn out to be more meaningful and produce the conjectured equality. As seen in the discussion leading up to Theorem \ref{sigmamodThm}, the sigma submodule satisfies
	\[
		\Sigma(\Theta; \field[\Delta])/\Theta = \sum_{j=1}^d \pi_j\left(\K^0(\hat{\Theta}_j, \theta_j)\right).
	\]
While the sigma submodule has been well-studied, it seems unnecessary to restrict ourselves to this submodule in particular when studying Hilbert series of reductions of Stanley--Reisner rings. Indeed, it seems very natural to instead consider the submodule
	\[
		\tau(\Theta; \field[\Delta]) = \sum_{j=1}^d\pi_j\left(\M^0(\hat{\Theta}_j)\right).
	\]
Note that in the calculations of $\Sigma(\Theta; \field[\Delta])_i$ that currently exist (those being the Buchsbaum case in \cite{MNY} and the $i=d-1$ case in this paper), trivial module structure has ensured that
	\[
		\tau(\Theta; \field[\Delta])_i=\Sigma(\Theta; \field[\Delta])_i
	\]
in the relevant degrees. However, this may not be true in greater generality. Furthermore, the proof techniques used in this paper suggest that computing the dimensions of graded pieces of the suggested $\tau(\Theta; \field[\Delta])$ may be more approachable than computing those of $\Sigma(\Theta; \field[\Delta])$ to begin with.

\subsection{Symmetries in $h^\sig(\Delta)$}

The following result concerning homology manifolds elegantly combines the algebraic computations of $h^\sig(\Delta)$, Poincar\'{e} duality, and Klee's combinatorial Dehn-Sommerville relations found in \cite{GrabeDS}.

\begin{theorem}\label{manifoldsymmetry}{\cite[Proposition 1.1]{MN-bdry}}
	Let $\Delta$ be a connected $(d-1)$-dimensional orientable homology manifold. Then
	\[
	h_i^\sig(\Delta) = h_{d-i}^\sig(\Delta)
	\]
	for all $0\le i\le d$.
\end{theorem}

 In a forthcoming paper of Sawaske and Xue, Klee's relations have been greatly generalized to account for all pure simplicial complexes through the following theorem (here $\mathbb{S}^i$ denotes a sphere of dimension $i$).

\begin{theorem}\label{DSThm}
	Let $\Delta$ be a pure $(d-1)$-dimensional simplicial complex. Then
	\[
	h_{d-j}(\Delta)-h_j(\Delta)=(-1)^j\sum_{F\in \Delta}\binom{d-|F|}{j}\left(\widetilde{\chi}(\lk_\Delta F)-\widetilde{\chi}(\mathbb{S}^{d-1-|F|})\right)
	\]
	for $j=0, \ldots, d$.
\end{theorem}

We are hopeful that combining these relations with a full computation of $h^\sig(\Delta)$ will allow for some version of symmetry to hold in analogy with Theorem \ref{manifoldsymmetry}. Already with the case covered in this paper, such symmetry can be stated as follows.

\begin{corollary}
	Let $\Delta$ be a pure and connected $(d-1)$-dimensional simplicial complex, and assume that the link of each vertex of $\Delta$ is connected as well. Then
	\[
	h_1^\sig(\Delta) - h_{d-1}^\sig(\Delta) = (-1)^{d-1}\sum_{F\in \Delta}\binom{d-|F|}{d-1}\left(\beta_{d-1-|F|}(\lk_\Delta F) - \beta_0(\lk_\Delta(F))\right).
	\]
\end{corollary}

Hence, the discrepancy between $h_1^\sig(\Delta)$ and $h_{d-1}^\sig(\Delta)$ measures the discrepancy in $\Delta$ satisfying Poincar\'{e} duality in the top and bottom dimensions. In fact, if $\Delta$ is also normal, then the use of intersection homology theory, in particular \cite{GM}[Theorem, Section 4.3], shows that this difference equivalently measures the accuracy of intersection homology.

\begin{corollary}
	Let $\Delta$ be a pure, normal, and connected $(d-1)$-dimensional simplicial complex, and assume that the link of each vertex of $\Delta$ is normal and connected as well. Then
	\[
	h_1^\sig(\Delta) - h_{d-1}^\sig(\Delta) = (-1)^{d-1}\sum_{F\in \Delta}\binom{d-|F|}{d-1}\left(\D(IH_0^{\overline{0}}(\lk_\Delta F)) - \beta_0(\lk_\Delta(F))\right).
	\]
\end{corollary}

\section*{Acknowledgments}
The author would like to thank Isabella Novik and Steve Klee, both of whom listened to multiple presentations of these results at all stages, offering valuable comments and suggestions throughout.

\bibliographystyle{alpha}
\bibliography{h-vect_and_lc-biblio}
\end{document}